\documentclass[11pt]{article}
\usepackage{amsmath,amssymb,amsthm}    
\usepackage{amsfonts,amscd} 
\usepackage[all]{xy} 

\makeatletter
\newfont{\footsc}{cmcsc10 at 8truept}
\newfont{\footbf}{cmbx10 at 8truept}  
\newfont{\footrm}{cmr10 at 10truept} 
\makeatother
\def\fl#1{\left\lfloor#1\right\rfloor}

\newtheorem{theorem}{Theorem}

\newtheorem{Cor}{Corollary}
\newtheorem{Lem}{Lemma}

\title{Continued fractions associated with the topological index of the caterpillar-bond graph}

\author{Takao Komatsu
\\
\small Department of Mathematics, School of Science\\
\small Zhejiang Sci-Tech University\\
\small Hangzhou 310018, China\\ 
\small \texttt{komatsu@zstu.edu.cn}
}

\date{
}

\allowdisplaybreaks 

\begin{document}
\maketitle

\begin{abstract}  
In this paper, we give graphs whose topological index are exactly equal to the number $u_n$, satisfying the three term recurrence relation 
$$
u_n=a u_{n-1}+b u_{n-2}\quad(n\ge 2)\quad u_0=0\quad\hbox{and}\quad u_1=u\,,   
$$  
where $a$, $b$ and $u$ are positive integers.  We show an interpretation from the continued fraction expansion in a more general case, so that the topological index can be computed easily.  On the contrary, for any given positive integer $N$, we can find the graphs (trees) whose topological indices are exactly equal to $N$.   
\end{abstract}

\section{Introduction} 

The concept of the {\it topological index} was first introduced by Haruo Hosoya in 1971 \cite{Hosoya1971}.  As more different types of topological indices have been discovered in chemical graph theory (e.g., see \cite{DB}), the first topological index is also called {\it Hosoya index} or the {\it $Z$ index} nowadays.  Topological indices are used for example in the development of quantitative structure-activity relationships (QSARs) in which the biological activity or other properties of molecules are correlated with their chemical structure.  
The integer $Z:=Z(G)$ is the sum of a set of the numbers $p(G,k)$, which is the number of ways for choosing $k$ disjoin edges from $G$. By using the set of $p(G,k)$, the topological index $Z$ is defined by 
$$
Z=\sum_{k=0}^m p(G,k)\,. 
$$   
The topological index is closely related to Fibonacci $F_n$ \cite{Hosoya1973} and related numbers \cite{Hosoya2007c}.  For the path graph $S_n$, we have $Z(S_n)=F_{n+1}$, where $F_n=F_{n-1}+F_{n-2}$ ($n\ge 2$) with $F_0=0$ and $F_1=1$.  For the monocyclic graph $C_n$, we have $Z(C_n)=L_n$, where $L_n$ is the Lucas number, defined by $L_n=L_{n-1}+L_{n-2}$ ($n\ge 2$) with $L_0=2$ and $L_1=1$.  
 
In \cite{Hosoya2007}, manipulation of continued fraction, either finite and infinite, was shown to be greatly simplified and systematized by introducing the topological index $Z$ and caterpillar graph $C_n(x_1,x_2,\dots,x_n)$. The continuant which was introduced by Euler in 18 century for solving continued fraction problems was shown to be identical to the $Z$-index of the caterpillar graph derived from the continued fraction concerned. Then the fastest algorithm for solving the Pell equations was obtained. Further, graph-theoretical interpretation for Fibonacci and Lucas numbers and generalized Fibonacci numbers was obtained.  
A {\it caterpillar graph} is a tree containing a path graph such that every edge has one or more endpoints in that path.   In \cite{Hosoya2007}, it is shown that for $n\ge 1$ 
\begin{equation}
Z\bigl(C_n(a_0,a_1,\dots,a_{n-1})\bigr)=p_{n-1}\,, 
\label{caterpillar-cf} 
\end{equation}
where 
$$
\frac{p_{n-1}}{q_{n-1}}=a_0+\cfrac{1}{a_1+{\atop\ddots+\cfrac{1}{a_{n-1}}}}\quad\hbox{with}\quad \gcd(p_{n-1},q_{n-1})=1,~ a_i\ge 1~(0\le i\le n-1)\,. 
$$ 
In \cite{Hosoya2008}, the three series of numbers, Fibonacci $F_n$, Lucas $L_n$ and generalized Fibonacci $G_n$ are defined to have the same recursive relation, $u_n=u_{n-1}+u_{n-2}$. By imposing the following set of initial conditions, $f_0=f_1=1$, $L_1=1$ and $L_2=3$, and $G_1=a>0$ and $G_2=b>0$ with $b>2 a$, a number of novel identities were found which systematically relate $f_n$, $L_n$, and $G_n$ with each other. Further, graph-theoretical interpretation for these relations was obtained by the aid of the continuant, caterpillar graph, and topological index $Z$ which was proposed and developed by Hosoya.  
In \cite{Hosoya2010}, the conventional algorithm for solving the linear Diophantine equation in two variables is greatly improved graph-theoretically by using the $Z$-caterpillars, namely, by substituting all the relevant series of integers with the caterpillar graphs whose topological indices represent those integers. By this graph-theoretical analysis, the mathematical structure of the linear Diophantine equation and its relation with Euclid's algorithm, continued fraction, and Euler's continuant is clarified. 

The numbers $u_n$, satisfying the three term recurrence relation $u_n=a u_{n-1}+u_{n-2}$, are entailed from the topological index of caterpillar graphs.  In particular, Pell numbers $P_n$, where $a=2$, are yielded from the comb graph \cite{Hosoya2007c}, which is the special case of the caterpillar graphs.  In addition, the numbers $u_n$ appear in the convergents $p_n/q_n$ of the simple continued fraction expansion.  
However, it does not seem that the numbers $u_n$, satisfying the three-term recurrence relation $u_n=u_{n-1}+b u_{n-2}$, have not been recognized as any special graph yet.  

In this paper, we give graphs whose topological index are exactly equal to the number $u_n$, satisfying the three-term recurrence relation 
$$
u_n=a u_{n-1}+b u_{n-2}\quad(n\ge 2)\quad u_0=0\quad\hbox{and}\quad u_1=u\,,   
$$  
where $a$, $b$ and $u$ are positive integers.  We show an interpretation from the continued fraction expansion in a more general case, so that the topological index can be computed easily.  On the contrary, for any given positive integer $N$, we can find the graphs (trees) whose topological indices are exactly equal to $N$.

\section{Double bonds}  

We explain double bonds in order to understand the structure of the sequence $\{u_n\}_{n\ge 0}$, satisfying the three term recurrence relation $u_n=u_{n-1}+2 u_{n-2}$.   
In Chemistry, double bonds are chemical bonds between two chemical elements involving four bonding electrons instead of the usual two, and found in ethylene (carbon-carbon double bond C=C), acetone (carbon-oxygen double bond C=O), dimethyl sulfoxide (sulfur-oxygen double bond S=O), diazene (nitrogen-nitrogen double bond N=N) and so on (see, e.g., \cite{March}).  
 
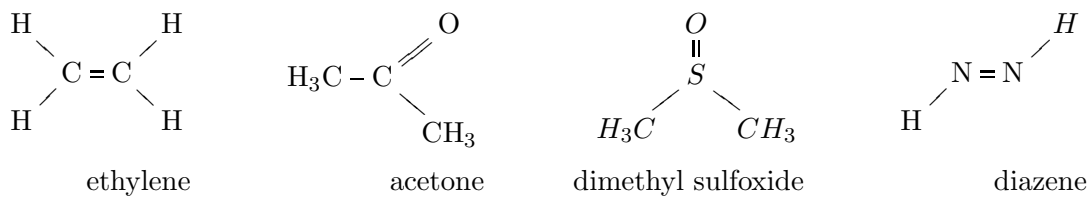
\begin{align*}
\xymatrix@=5pt{
{\rm H}\ar@{-}[rd]&&&{\rm H}\\
&{\rm C}\ar@{=}[r]&{\rm C}\ar@{-}[ru]\ar@{-}[rd]&\\
{\rm H}\ar@{-}[ru]&&&{\rm H}\\ 
}& \quad &
\xymatrix@C=5pt@R=6pt{
&&{\rm O}\\
{\rm H_3 C}\ar@{-}[r]&{\rm C}\ar@{=}[ur]\ar@{-}[rd]&\\ 
&&{\rm C H_3}\\
}& \quad &
\xymatrix@C=5pt@R=6pt{
&O\ar@{=}[d]&\\
&S&\\
H_3 C\ar@{-}[ru]&&C H_3\ar@{-}[lu]\\
}& \quad &
\xymatrix@=5pt{
&&&H\\
&{\rm N}\ar@{=}[r]&{\rm N}\ar@{-}[ru]\\  
{\rm H}\ar@{-}[ru]&&&
}\\*
\text{ethylene}&&\text{acetone}&&\text{dimethyl sulfoxide}&&\text{diazene}
\end{align*}
 
Though there does not seem to exist any concrete example, we shall consider the connected graph of double bonds as $B_n$.   
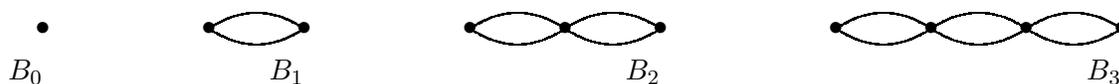
\begin{align*}
\xymatrix@=36pt{*=0{\bullet}}&~&    
\xymatrix@=36pt{&*=0{\bullet}\ar @/^/ @{-} [r] 
\ar @/_/ @{-} [r]&*=0{\bullet} 
}&~&  
\xymatrix@=36pt{&*=0{\bullet}\ar @/^/ @{-} [r] 
\ar @/_/ @{-} [r]&*=0{\bullet}
\ar @/^/ @{-} [r] 
\ar @/_/ @{-} [r]&*=0{\bullet} 
}&~&~  
\xymatrix@=36pt{&*=0{\bullet}\ar @/^/ @{-} [r] 
\ar @/_/ @{-} [r]&*=0{\bullet}
\ar @/^/ @{-} [r] 
\ar @/_/ @{-} [r]&*=0{\bullet}
\ar @/^/ @{-} [r] 
\ar @/_/ @{-} [r]&*=0{\bullet} 
}
\\*
B_0&&B_1&&B_2&&B_3 
\end{align*}  

Then the topological index of $B_n$ coincides with the Jacobsthal number, whose sequence is given by 
$$
0, 1, 1, 3, 5, 11, 21, 43, 85, 171, 341, 683, 1365, 2731, 5461, 10923, 21845, 43691, 87381, 174763, \dots 
$$
(\cite[A001045]{oeis}).  

\begin{theorem}  
For $n\ge 0$ 
$$  
Z(B_n)=J_{n+2},  
$$ 
where $J_n$ are the Jacobsthal numbers defined by  
$$
J_n=J_{n-1}+2 J_{n-2}\quad(n\ge 2)\quad\hbox{with}\quad J_0=0\quad\hbox{and}\quad J_1=1\,. 
$$ 
\label{jacobsthal}  
\end{theorem}   

Theorem \ref{jacobsthal} is a special case of the main result in the later section.  
Theorem \ref{jacobsthal} holds for small $n$ by the following table.   

\begin{center}  
\begin{table}[htb]
\begin{tabular}{r|c|c|c|c} 
&$k=0$&$k=1$&$k=2$&$Z(B_n)$\\\hline 
$p(B_0,k)$&$1$&&&$1$\\ 
$p(B_1,k)$&$1$&$2$&&$3$\\ 
$p(B_2,k)$&$1$&$4$&&$5$\\ 
$p(B_3,k)$&$1$&$6$&$4$&$11$
\end{tabular}
\end{table}
\end{center}

There exist stronger bonds in chemistry.  
Triple bonds are of order $3$. Some chemical compounds with a triple bond are acetylene and cyanogen.  

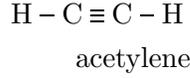
\begin{align*}
\xymatrix@=5pt{
{\rm H}\ar@{-}[r]&{\rm C}\ar@3{-}[r]&{\rm C}\ar@{-}[r]&{\rm H}
}& \quad &
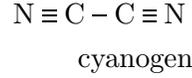
\xymatrix@=5pt{
{\rm N}\ar@3{-}[r]&{\rm C}\ar@{-}[r]&{\rm C}\ar@3{-}[r]&{\rm N}
}\\*
\text{acetylene}&&\text{cyanogen}
\end{align*} 

Quadruple bond (e.g., chromium(II) acetate), Quintuple bond and Sextuple bond have been also known as of order $4$, $5$ and $6$, respectively.  

In Mathematics, define a {\it bond graph} denoted by $B_{n}(y_1,y_2,\dots,y_{n-1})$, a connected graph with bonds order $y_1$, $y_2$, $\dots$, $y_{n-1}$, where $y_1,y_2,\dots,y_{n-1}$ are positive integers.   
$$
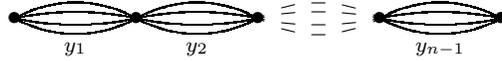
\xymatrix@=46pt{
*=0{\bullet}\ar @/^/ @{-} [r]\ar @/^0.2pc/ @{-} [r] \ar @/_0.2pc/ @{-} [r] 
\ar @/_/ @{-} [r]_{y_1}&*=0{\bullet}
\ar @/^/ @{-} [r]\ar @/^0.2pc/ @{-} [r] \ar @/_0.2pc/ @{-} [r] 
\ar @/_/ @{-} [r]_{y_2}&*=0{\bullet}
\ar @/^/ @{--} [r]\ar @/^0.2pc/ @{--} [r] \ar @/_0.2pc/ @{--} [r] 
\ar @/_/ @{--} [r]&*=0{\bullet}
\ar @/^/ @{-} [r]\ar @/^0.2pc/ @{-} [r] \ar @/_0.2pc/ @{-} [r] 
\ar @/_/ @{-} [r]_{y_{n-1}}&*=0{\bullet}
}  
$$ 
If $y_1=y_2=\cdots=y_{n-1}=1$, $S_n=B_n(1,1,\dots,1)$ is the path graph.  
If $y_1=y_2=\cdots=y_{n-1}=2$, $B_n=B_n(2,2,\dots,2)$ yields Jacobsthal numbers in its topological indices above.  
Similarly, if $y_1=y_2=\cdots=y_{n-1}=b$, $B_n(b,b,\dots,b)$ is related with the number $u_n$, satisfying the three term recurrence relation $u_n=u_{n-1}+b u_{n-2}$ ($n\ge 3$) with $u_1=b$ and $u_2=b+1$.  
In fact, we shall discuss more general cases in the later section.

\section{Continued fraction}  

Any real number $\alpha$ is expressed as the regular (or simple) continued fraction expansion 
$$  
\alpha:=[a_0;a_1,a_2,\dots]
=a_0+\cfrac{1}{a_1+\cfrac{1}{a_2+{\atop\ddots}}}\,, 
$$  
where 
\begin{eqnarray*}
\alpha&=&a_0+\theta_0,\quad
a_0=\fl{\alpha},\\
1/\theta_{n-1}&=&a_n+\theta_n,\quad 
a_n=\fl{1/\theta_{n-1}}\quad(n\ge 1)\,. 
\end{eqnarray*} 
The $n$-th convergent of the continued fraction expansion of $\alpha$ is given by 
$$   
\frac{p_n}{q_n}:=[a_0;a_1,a_2,\dots,a_n]  
=a_0+\cfrac{1}{a_1+\cfrac{1}{a_2+{\atop\ddots{\atop +\dfrac{1}{a_n}}}}}\,. 
$$ 
It is well-known that $p_n$ and $q_n$ satisfy the recurrence relation: 
\begin{align}
p_n&=a_n p_{n-1}+p_{n-2}\quad(n\ge 0),& p_{-1}&=1, &p_{-2}&=0,
\label{recurrence-gp}\\
q_n&=a_n q_{n-1}+q_{n-2}\quad(n\ge 0),& q_{-1}&=0, &q_{-2}&=1\,.
\label{recurrence-gq}
\end{align}

In Graph theory, a caterpillar graph (or tree), denoted by $C_n(x_1,x_2,\dots,x_n)$, is a tree in which all the vertices are within distance 1 of a central path.   

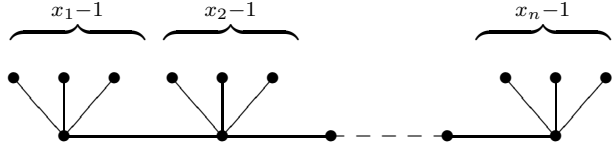
\begin{align*} 
&\overbrace{\phantom{\hspace{0.7in}}}^{x_1-1}~ 
\overbrace{\phantom{\hspace{0.7in}}}^{x_2-1}\qquad\qquad\quad\quad  
\overbrace{\phantom{\hspace{0.7in}}}^{x_n-1}\\* 
&\xymatrix@=16pt{ 
*=0{\bullet}&*=0{\bullet}&*=0{\bullet}&*=0{\bullet}&*=0{\bullet}&*=0{\bullet}&&&&*=0{\bullet}&*=0{\bullet}&*=0{\bullet}\\
&*=0{\bullet}\ar @{-} [lu]\ar @{-} [u]\ar @{-} [ru]\ar @{-} [rrr]&&&*=0{\bullet}\ar @{-} [lu]\ar @{-} [u]\ar @{-} [ru]\ar @{-} [rr]&&*=0{\bullet}\ar @{--} [rr]&&*=0{\bullet}&&*=0{\bullet}\ar @{-} [lu]\ar @{-} [u]\ar @{-} [ru]\ar @{-} [ll]&\\
}
\end{align*} 


If $x_1=\cdots=x_n=1$, $S_n=C_n(1,\dots,1)$ is a path graph.  

In \cite{Hosoya2007}, it is shown that for $n\ge 1$ 
\begin{equation}
Z\bigl(C_n(a_0,a_1,\dots,a_{n-1}\bigr)=p_{n-1}\,, 
\label{caterpillar-cf} 
\end{equation}
where 
$$
\frac{p_{n-1}}{q_{n-1}}=a_0+\cfrac{1}{a_1+{\atop\ddots+\cfrac{1}{a_{n-1}}}}\quad\hbox{with}\quad \gcd(p_{n-1},q_{n-1})=1,~ a_i\ge 1~(0\le i\le n-1)\,. 
$$

\section{Main results} 

Any real number can be expressed as a generalized continued fraction expansion of the form 
$$ 
\alpha
=a_0+\cfrac{b_1}{a_1+\cfrac{b_2}{a_2+{\atop\ddots}}}\,.  
$$  
In this paper, we assume that all numbers $a_0,a_1,a_2,\dots$ and $b_1,b_2,\dots$ are positive integers.  
The $n$-th convergent $p_n/q_n$ is given by 
$$   
\frac{p_n}{q_n}
=a_0+\cfrac{b_1}{a_1+\cfrac{b_2}{a_2+{\atop\ddots{\atop +\dfrac{b_n}{a_n}}}}}\,. 
$$ 
Here, $p_n$ and $q_n$ satisfy the recurrence relation: 
\begin{align*}
p_n&=a_n p_{n-1}+b_n p_{n-2}\quad(n\ge 2),& p_{0}&=a_0, &p_{1}&=a_0 a_1+b_1,\\
q_n&=a_n q_{n-1}+b_n q_{n-2}\quad(n\ge 2),& q_{0}&=1, &q_{1}&=a_1\,.
\end{align*}

Notice that the expression of the generalized continued fraction expansion is not unique, and $p_n$ and $q_n$ are not necessarily coprime.  

Now, we introduce a combined graph of the caterpillar graph and the bond graph as their generalization.  

Caterpillar-bond graph  $D_n(x_1,x_2,\dots,x_n;y_1,y_2,\dots,y_{n-1})$

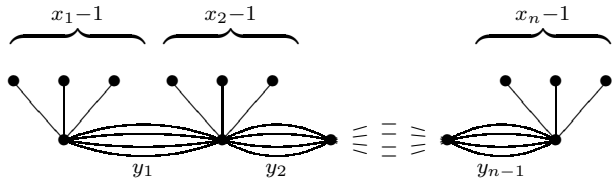
\begin{align*} 
&\overbrace{\phantom{\hspace{0.7in}}}^{x_1-1}~ 
\overbrace{\phantom{\hspace{0.7in}}}^{x_2-1}\qquad\qquad\quad\quad  
\overbrace{\phantom{\hspace{0.7in}}}^{x_n-1}\\* 
&\xymatrix@=16pt{ 
*=0{\bullet}&*=0{\bullet}&*=0{\bullet}&*=0{\bullet}&*=0{\bullet}&*=0{\bullet}&&&&*=0{\bullet}&*=0{\bullet}&*=0{\bullet}\\
&*=0{\bullet}\ar @{-} [lu]\ar @{-} [u]\ar @{-} [ru]\ar @/^/ @{-} [rrr]\ar @/^0.2pc/ @{-} [rrr] \ar @/_0.2pc/ @{-} [rrr]  
\ar @/_/ @{-} [rrr]_{y_1} &&&*=0{\bullet}\ar @{-} [lu]\ar @{-} [u]\ar @{-} [ru]\ar @/^/ @{-} [rr]\ar @/^0.2pc/ @{-} [rr] \ar @/_0.2pc/ @{-} [rr] 
\ar @/_/ @{-} [rr]_{y_2}&&*=0{\bullet}\ar @/^/ @{--} [rr]\ar @/^0.2pc/ @{--} [rr] \ar @/_0.2pc/ @{--} [rr]  
\ar @/_/ @{--} [rr]&&*=0{\bullet}&&*=0{\bullet}\ar @{-} [lu]\ar @{-} [u]\ar @{-} [ru]\ar @/^/ @{-} [ll]^{y_{n-1}}\ar @/^0.2pc/ @{-} [ll] \ar @/_0.2pc/ @{-} [ll]\ar @/_/ @{-} [ll]&\\
}
\end{align*} 

{\bf Example I.}  
For example, the caterpillar-bond graph $D_4(3,1,2,4;3,4,1)$ is given by the following.  
$$
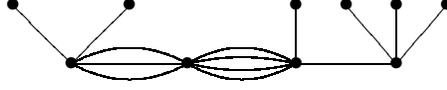
\xymatrix@=16pt{
*=0{\bullet}&&*=0{\bullet}&&&*=0{\bullet}&*=0{\bullet}&*=0{\bullet}&*=0{\bullet}\\
&*=0{\bullet}\ar @{-} [lu]\ar @{-} [ru]\ar @/^/ @{-} [rr]\ar @{-} [rr]  
\ar @/_/ @{-} [rr]&&*=0{\bullet}\ar @/^/ @{-} [rr]\ar @/^0.2pc/ @{-} [rr] \ar @/_0.2pc/ @{-} [rr] \ar @/_/ @{-} [rr]&&*=0{\bullet}\ar @{-} [u]\ar @{-} [rr]  
&&*=0{\bullet}\ar @{-} [lu]\ar @{-} [u]\ar @{-} [ru]&
}  
$$    

Notice that 
\begin{align}  
&D_n(2,x_2,\dots,x_n;y_1,\dots,y_{n-1})=D_{n+1}(1,1,x_2,\dots,x_n;1,y_1,\dots,y_{n-1})\,,
\label{eq:dn1}\\ 
&D_n(x_1,\dots,x_{n-1},2;y_1,\dots,y_{n-1})=D_{n+1}(x_1,\dots,x_{n-1},1,1;y_1,\dots,y_{n-1},1)\,. 
\label{eq:dn2}
\end{align}

Our main result can be stated as follows.  

\begin{theorem}  
For $n\ge 1$, 
$$
Z\bigl(D_n(a_0,a_1,\dots,a_{n-1};b_1,\dots,b_{n-1})\bigr)=p_{n-1}\,. 
$$ 
\label{th:main}  
\end{theorem} 

{\it Remark.}  
Since 
$$
{\ddots\atop}+\cfrac{b}{1+\cfrac{1}{1}}={\ddots\atop}+\frac{b}{2} 
$$ 
with two continued fraction expansions of $p/q$ and $p/(p-q)$ ($p>q$),  
we can recognize the relations (\ref{eq:dn1}) and (\ref{eq:dn2}), and their topological indices are the same.  

{\bf Example II.}  
Since 
$$
3+\cfrac{3}{1+\cfrac{4}{2+\cfrac{1}{4}}}=\frac{102}{25}\,, 
$$
the topological index is given by $Z\bigl(D_4(3,1,2,4;3,4,1)\bigr)=102$.

In order to prove our main result, we need the known relations, which were first suggested by Hosoya \cite{Hosoya1971,Hosoya1973} and were elaborated by Gutman and Polansky \cite{GP}.  Though we need only the first one in this paper, we also list related relations for convenience.   

\begin{Lem}  
\begin{enumerate}  
\item If $e=u v$ is an edge of a graph $G$, then $Z(G)=Z(G-e)+Z(G-\{u,v\})$.  
\item If $v$ is a vertex of a graph $G$, then $Z(G)=Z(G-v)+\sum_{u v}Z(G-u v)$, where the summation extends over all vertices adjacent to $v$.   
\item If $G_1,G_2,\dots,G_k$ are connected components of $G$, then $Z(G)=\prod_{i=1}^k Z(G_i)$.  
\end{enumerate} 
\label{gp} 
\end{Lem}

\begin{proof}[Proof of Theorem \ref{th:main}]  
We can show that for $n\ge 3$ 
\begin{align}  
&Z\bigl(D_n(x_1,x_2,\dots,x_n;y_1,\dots,y_{n-1})\bigr)\notag\\
&=x_n Z\bigl(D_{n-1}(x_1,x_2,\dots,x_{n-1};y_1,\dots,y_{n-2})\bigr)
 +y_{n-1}Z\bigl(D_{n-2}(x_1,x_2,\dots,x_{n-2};y_1,\dots,y_{n-3})\bigr)\,. 
\label{recurrence-zd}
\end{align}  
Manually, we can compute 
$$
Z\bigl(D_1(x_1)\bigr)=x_1\quad\hbox{and}\quad Z\bigl(D_2(x_1,x_2,y_1)\bigr)=x_1 x_2+y_1\,. 
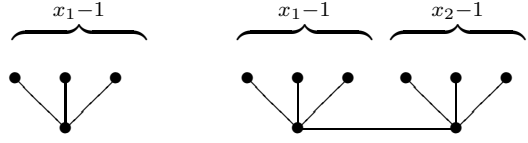
$$ 
\begin{align*} 
&\overbrace{\phantom{\hspace{0.7in}}}^{x_1-1}\qquad\quad 
\overbrace{\phantom{\hspace{0.7in}}}^{x_1-1}~ 
\overbrace{\phantom{\hspace{0.7in}}}^{x_2-1}
\\* 
&\xymatrix@=16pt{ 
*=0{\bullet}&*=0{\bullet}&*=0{\bullet}\\
&*=0{\bullet}\ar @{-} [lu]\ar @{-} [u]\ar @{-} [ru]&
}\qquad\qquad
\xymatrix@=16pt{ 
*=0{\bullet}&*=0{\bullet}&*=0{\bullet}&*=0{\bullet}&*=0{\bullet}&*=0{\bullet}\\
&*=0{\bullet}\ar @{-} [lu]\ar @{-} [u]\ar @{-} [ru]\ar @{-} [rrr]&&&*=0{\bullet}\ar @{-} [lu]\ar @{-} [u]\ar @{-} [ru]&\\
}
\end{align*}  
On the other hand, for this general continued fraction expansion, we know that   
$$
a_0=\frac{a_0}{1}=\frac{p_0}{q_0}\quad\hbox{and}\quad a_0+\frac{b_1}{a_1}=\frac{a_0 a_1+b_1}{a_1}=\frac{p_1}{q_1}
$$ 
with the recurrence relation (\ref{recurrence-gp}).    
By setting $x_k=a_{k-1}$ ($k\ge 1$) and $y_k=b_k$ ($k\ge 1$), the structures of $Z\bigl(D_n(x_1,x_2,\dots,x_n;y_1,\dots,y_{n-1})\bigr)$ and $p_{n-1}$ are completely the same.  Therefore, we obtain $Z\bigl(D_n(x_1,x_2,\dots,x_n;y_1,\dots,y_{n-1})\bigr)=p_{n-1}$ ($n\ge 1$).    
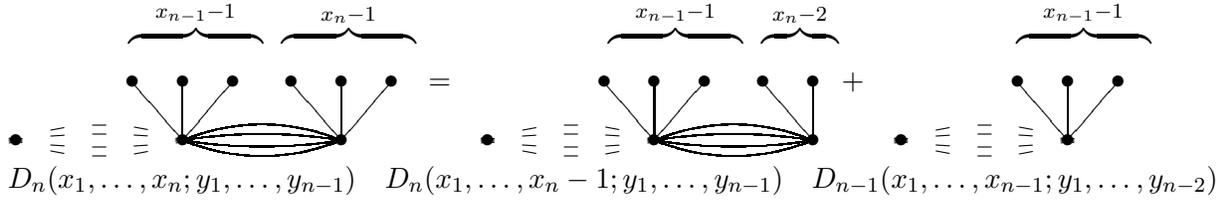
\begin{align*}  
&\qquad\qquad\overbrace{\phantom{\hspace{0.7in}}}^{x_{n-1}-1}~ 
\overbrace{\phantom{\hspace{0.7in}}}^{x_n-1}\qquad\qquad\qquad~  
\overbrace{\phantom{\hspace{0.7in}}}^{x_{n-1}-1}~ 
\overbrace{\phantom{\hspace{0.4in}}}^{x_n-2}\qquad\qquad\qquad \overbrace{\phantom{\hspace{0.7in}}}^{x_{n-1}-1}\\*
&\xymatrix@=16pt{ 
&&*=0{\bullet}&*=0{\bullet}&*=0{\bullet}&*=0{\bullet}&*=0{\bullet}&*=0{\bullet}\\
*=0{\bullet}&&&*=0{\bullet}\ar @{-} [lu]\ar @{-} [u]\ar @{-} [ru]\ar @/^/ @{-}[rrr]\ar @/^0.2pc/ @{-} [rrr] \ar @/_0.2pc/ @{-} [rrr]\ar @/_/ @{-}[rrr]\ar @/^/ @{--}[lll]\ar @/^0.2pc/ @{--} [lll] \ar @/_0.2pc/ @{--} [lll]\ar @/_/ @{--}[lll]&&&*=0{\bullet}\ar @{-} [lu]\ar @{-} [u]\ar @{-} [ru]&
}
\quad {\mathbf =} \quad 
\xymatrix@=16pt{ 
&&*=0{\bullet}&*=0{\bullet}&*=0{\bullet}&*=0{\bullet}&*=0{\bullet}\\
*=0{\bullet}&&&*=0{\bullet}\ar @{-} [lu]\ar @{-} [u]\ar @{-} [ru]\ar @/^/ @{-}[rrr]\ar @/^0.2pc/ @{-} [rrr] \ar @/_0.2pc/ @{-} [rrr]\ar @/_/ @{-}[rrr]\ar @/^/ @{--}[lll]\ar @/^0.2pc/ @{--} [lll] \ar @/_0.2pc/ @{--} [lll]\ar @/_/ @{--}[lll]&&&*=0{\bullet}\ar @{-} [lu]\ar @{-} [u]
}
\quad {\mathbf +} \quad 
\xymatrix@=16pt{ 
&&*=0{\bullet}&*=0{\bullet}&*=0{\bullet}\\
*=0{\bullet}&&&*=0{\bullet}\ar @{-} [lu]\ar @{-} [u]\ar @{-} [ru]\ar @/^/ @{--}[lll]\ar @/^0.2pc/ @{--} [lll] \ar @/_0.2pc/ @{--} [lll]\ar @/_/ @{--}[lll]&
}\\*
&D_n(x_1,\dots,x_n;y_1,\dots,y_{n-1}) \quad D_n(x_1,\dots,x_n-1;y_1,\dots,y_{n-1})  
\quad D_{n-1}(x_1,\dots,x_{n-1};y_1,\dots,y_{n-2})  
\end{align*} 
Finally, we prove (\ref{recurrence-zd}). 
By using the first relation in Lemma \ref{gp} repeatedly, 
\begin{align*}    
&Z\bigl(D_n(x_1,x_2,\dots,x_n;y_1,\dots,y_{n-1})\bigr)\\  
&=Z\bigl(D_{n-1}(x_1,x_2,\dots,x_{n-1};y_1,\dots,y_{n-2})\bigr) 
+Z\bigl(D_n(x_1,x_2,\dots,x_n-1;y_1,\dots,y_{n-1})\bigr)\\
&=\cdots\\
&=(x_n-1)Z\bigl(D_{n-1}(x_1,x_2,\dots,x_{n-1};y_1,\dots,y_{n-2})\bigr) 
+Z\bigl(D_n(x_1,x_2,\dots,1;y_1,\dots,y_{n-1})\bigr)\\  
&=(x_n-1)Z\bigl(D_{n-1}(x_1,x_2,\dots,x_{n-1};y_1,\dots,y_{n-2})\bigr)\\
&\quad +Z\bigl(D_{n-2}(x_1,x_2,\dots,x_{n-2};y_1,\dots,y_{n-3})\bigr)\\ 
&\quad +Z\bigl(D_n(x_1,x_2,\dots,x_{n-1},1;y_1,\dots,y_{n-2},y_{n-1}-1)\bigr)\\  
&\cdots\\
&=(x_n-1)Z\bigl(D_{n-1}(x_1,x_2,\dots,x_{n-1};y_1,\dots,y_{n-2})\bigr)\\ 
&\quad +(y_{n-1}-1)Z\bigl(D_{n-2}(x_1,x_2,\dots,x_{n-2};y_1,\dots,y_{n-3})\bigr)\\ 
&\quad +Z\bigl(D_n(x_1,x_2,\dots,x_{n-1},1;y_1,\dots,y_{n-2},1)\bigr)\\  
&=x_n Z\bigl(D_{n-1}(x_1,x_2,\dots,x_{n-1};y_1,\dots,y_{n-2})\bigr) 
+y_{n-1}Z\bigl(D_{n-2}(x_1,x_2,\dots,x_{n-2};y_1,\dots,y_{n-3})\bigr)\,. 
\end{align*}  
\end{proof}   

\begin{proof}[Additional proof]  
We can recognize the desired result by a tridiagonal determinantal expression.  
\begin{align*}  
&K_n(x_1,\dots,x_n;y_1,\dots,y_{n-1}):=
\left|\begin{array}{ccccc}
x_1&y_1&0&&\\
-1&x_2&y_2&\ddots&\\
0&-1&\ddots&\ddots&0\\
&\ddots&\ddots&x_{n-1}&y_{n-1}\\ 
&&0&-1&x_n 
\end{array}\right|\\
&=x_n\left|\begin{array}{ccccc}
x_1&y_1&0&&\\
-1&x_2&y_2&\ddots&\\
0&-1&\ddots&\ddots&0\\
&\ddots&\ddots&x_{n-2}&y_{n-2}\\ 
&&0&-1&x_{n-1} 
\end{array}\right| 
 -y_{n-1}\left|\begin{array}{ccccc}
x_1&y_1&0&&\\
-1&x_2&y_2&\ddots&\\
0&\ddots&\ddots&\ddots&0\\
&&-1&x_{n-2}&y_{n-2}\\ 
0&\cdots&\cdots&0&-1  
\end{array}\right|\\
&=x_n K_n(x_1,\dots,x_n;y_1,\dots,y_{n-1})+y_{n-1}K_{n-2}(x_1,\dots,x_{n-2};y_1,\dots,y_{n-3})
\end{align*}
with 
$$
K_1(x_1)=|x_1|=x_1\quad\hbox{and}\quad 
K_2(x_1,x_2;y_1)=\left|\begin{array}{rr} 
x_1&y_1\\
-1&x_2
\end{array}\right|=x_1 x_2+y_1\,. 
$$ 
\end{proof}

\subsection{Special cases with recurrence relations}  

If $x_1=\cdots=x_n=a$ and $y_1=\cdots=y_{n-1}=b$ in Theorem \ref{th:main}, we have the following.  

\begin{Cor}  
Let $a$ and $b$ be positive integers.   Then 
for a positive integer $n$,   
\begin{align*}
Z\bigl(D_n(\underbrace{a,\dots,a}_n;\underbrace{b,\dots,b}_{n-1})\bigr)&=u_{n+1}\\
&=a u_n+b u_{n-1}
\end{align*} 
with $u_0=0$ and $u_1=1$.   
\label{cor1} 
\end{Cor}   

If the initial values are also arbitrary, then we have the following.  

\begin{Cor}  
For a positive integer $n$,   
\begin{align*}
Z\bigl(D_n(v_1,\underbrace{a,\dots,a}_{n-1};b v_0,\underbrace{b,\dots,b}_{n-2})\bigr)&=v_{n}\\
&=a v_{n-1}+b v_{n-2}\quad(n\ge 2)\,. 
\end{align*}   
\label{cor2} 
\end{Cor}  

More specific cases are for Fibonacci $F_n$, Lucas $L_n$, Pell $P_n$, Pell-Lucas $Q_n$ and Jacobsthal numbers $J_n$, where
\begin{align*} 
&F_n=F_{n-1}+F_{n-2}\quad(n\ge 2)\quad\hbox{with}\quad F_0=0\quad\hbox{and}\quad F_1=1\,,\\
&L_n=L_{n-1}+L_{n-2}\quad(n\ge 2)\quad\hbox{with}\quad L_0=2\quad\hbox{and}\quad L_1=1\,,\\
&P_n=2 P_{n-1}+P_{n-2}\quad(n\ge 2)\quad\hbox{with}\quad P_0=0\quad\hbox{and}\quad P_1=1\,,\\
&Q_n=2 Q_{n-1}+Q_{n-2}\quad(n\ge 2)\quad\hbox{with}\quad Q_0=2\quad\hbox{and}\quad Q_1=2\,,\\
&J_n=J_{n-1}+2 J_{n-2}\quad(n\ge 2)\quad\hbox{with}\quad J_0=0\quad\hbox{and}\quad J_1=1\,. 
\end{align*}  

\begin{align*}  
Z\bigl(D_n(\underbrace{1,\dots,1}_n;\underbrace{1,\dots,1}_{n-1})\bigr)&=F_{n+1}\,,\\
Z\bigl(D_n(\underbrace{1,\dots,1}_n;2,\underbrace{1,\dots,1}_{n-2})\bigr)&=L_{n}\,,\\
Z\bigl(D_n(\underbrace{2,\dots,2}_n;\underbrace{1,\dots,1}_{n-1})\bigr)&=P_{n+1}\,,\\
Z\bigl(D_n(\underbrace{2,\dots,2}_n;2,\underbrace{1,\dots,1}_{n-2})\bigr)&=Q_{n+1}\,,\\
Z\bigl(D_n(\underbrace{1,\dots,1}_n;\underbrace{2,\dots,2}_{n-1})\bigr)&=Z\bigl(B_{n-1}\bigr)=J_{n+1}\,. 
\end{align*} 

The first four cases can be seen in \cite{Hosoya2007c,Hosoya2007,Hosoya2008}. The last case is exactly the same as Theorem \ref{jacobsthal}.  In \cite{Hosoya2007c} more numbers with corresponding graphs are presented, and graphs of $L_n$ and $Q_n$ are different.  Another graph of $L_n$ by Hosoya is the monocyclic graph $C_n$, where $Z(C_n)=L_n$.  

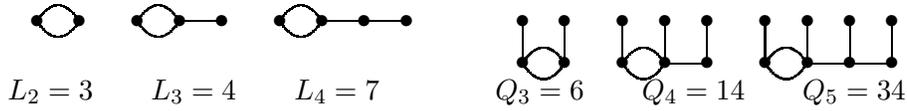
\begin{align*} 
&\quad  
\xymatrix@=16pt{ 
*=0{\bullet}\ar@/^/@{-}[r]\ar@/_/@{-}[r]&*=0{\bullet}
}\qquad 
\xymatrix@=16pt{ 
*=0{\bullet}\ar@/^/@{-}[r]\ar@/_/@{-}[r]&*=0{\bullet}\ar@{-}[r]&*=0{\bullet}
}\qquad 
\xymatrix@=16pt{ 
*=0{\bullet}\ar@/^/@{-}[r]\ar@/_/@{-}[r]&*=0{\bullet}\ar@{-}[r]&*=0{\bullet}\ar@{-}[r]&*=0{\bullet}
}\qquad\qquad  
\xymatrix@=16pt{
*=0{\bullet}\ar@{-}[d]&*=0{\bullet}\ar@{-}[d]\\ 
*=0{\bullet}\ar@/^/@{-}[r]\ar@/_/@{-}[r]&*=0{\bullet}
}\qquad 
\xymatrix@=16pt{ 
*=0{\bullet}\ar@{-}[d]&*=0{\bullet}\ar@{-}[d]&*=0{\bullet}\ar@{-}[d]\\ 
*=0{\bullet}\ar@/^/@{-}[r]\ar@/_/@{-}[r]&*=0{\bullet}\ar@{-}[r]&*=0{\bullet}
}\qquad 
\xymatrix@=16pt{ 
*=0{\bullet}\ar@{-}[d]&*=0{\bullet}\ar@{-}[d]&*=0{\bullet}\ar@{-}[d]&*=0{\bullet}\ar@{-}[d]\\
*=0{\bullet}\ar@/^/@{-}[r]\ar@/_/@{-}[r]&*=0{\bullet}\ar@{-}[r]&*=0{\bullet}\ar@{-}[r]&*=0{\bullet}
}\\*
& L_2=3\qquad L_3=4\qquad L_4=7\qquad\qquad Q_3=6\qquad Q_4=14\qquad Q_5=34 
\end{align*}

\subsection{Applications}    
    
Using the continued fraction expansion, we can compute the topological index of the graph by Theorem \ref{th:main}.  

On the other hand, we can constitute the graph (without any ring) whose topological index is given.  For example, we shall find the graphs whose topological index are $17$.  Then, concerning the continued fractions we get 
\begin{align*}  
&17,\quad \frac{17}{2}=8+\frac{1}{2},\quad \frac{17}{3}=5+\frac{2}{3},\quad \frac{17}{4}=4+\frac{1}{4},\quad \frac{17}{5}=3+\frac{2}{5},\quad \frac{17}{6}=2+\frac{5}{6},\\ 
&\frac{17}{7}=2+\frac{3}{7},\quad \frac{17}{8}=2+\frac{1}{8},\quad \frac{17}{9}=1+\cfrac{1}{1+\cfrac{1}{8}},\quad \frac{17}{10}=1+\cfrac{1}{1+\cfrac{3}{7}},\quad \frac{17}{11}=1+\cfrac{1}{1+\cfrac{5}{6}},\\
&\frac{17}{12}=1+\cfrac{1}{2+\cfrac{2}{5}},\quad \frac{17}{13}=1+\cfrac{1}{3+\cfrac{1}{4}},\quad \frac{17}{14}=1+\cfrac{1}{4+\cfrac{2}{3}},\quad \frac{17}{15}=1+\cfrac{1}{7+\cfrac{1}{2}},\quad \frac{17}{16}=1+\frac{1}{16}\,.
\end{align*} 
If we allow (\ref{eq:dn1}) and (\ref{eq:dn2}), we still have different expressions with the same value. For example,  
$$
\frac{17}{3}=5+\frac{2}{3}=5+\cfrac{1}{1+\cfrac{1}{2}}=5+\cfrac{1}{1+\cfrac{1}{1+\cfrac{1}{1}}}
$$ 
and 
$$
\frac{17}{14}=1+\cfrac{3}{14}=1+\cfrac{1}{4+\cfrac{2}{3}}=1+\cfrac{1}{4+\cfrac{1}{1+\cfrac{1}{2}}}=1+\cfrac{1}{4+\cfrac{1}{1+\cfrac{1}{1+\cfrac{1}{1}}}}\,.  
$$ 
However, the graph structures of $\frac{17}{3}$ and $\frac{17}{14}$ are essentially the same.  It is similar for $\frac{17}{q}$ and $\frac{17}{17-q}$.  Therefore, the essentially different graphs whose topological indices are equal to $17$ are given as follows.  
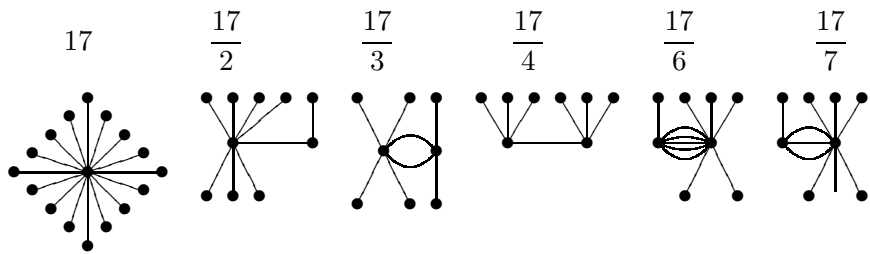
\begin{align*}  
&\qquad 17\qquad\qquad\frac{17}{2}\qquad\qquad\frac{17}{3}\qquad\qquad\frac{17}{4}\qquad\qquad\frac{17}{6}\qquad\qquad\frac{17}{7}\\* 
&\xymatrix@=1pt{  
&&&&*=0{\bullet}&&&&\\
&&&*=0{\bullet}&&*=0{\bullet}&&&\\
&&*=0{\bullet}&&&&*=0{\bullet}&&\\
&*=0{\bullet}&&&&&&*=0{\bullet}&\\
*=0{\bullet}&&&&*=0{\bullet}\ar @{-} [lluu]\ar @{-} [luuu]\ar @{-} [uuuu]\ar @{-} [ruuu]\ar @{-} [rruu]\ar @{-} [rrru]\ar @{-} [rrrr]\ar @{-} [rrrd]\ar @{-} [rrdd]\ar @{-} [rddd]\ar @{-} [dddd]\ar @{-} [lddd]\ar @{-} [lldd]\ar @{-} [llld]\ar @{-} [llll]\ar @{-} [lllu]&&&&*=0{\bullet}\\
&*=0{\bullet}&&&&&&*=0{\bullet}&\\
&&*=0{\bullet}&&&&*=0{\bullet}&&\\
&&&*=0{\bullet}&&*=0{\bullet}&&&\\
&&&&*=0{\bullet}&&&&
}\quad 
\xymatrix@=4pt{  
*=0{\bullet}&*=0{\bullet}&*=0{\bullet}&*=0{\bullet}&*=0{\bullet}\\
&&&&\\
&*=0{\bullet}\ar@{-}[luu]\ar@{-}[uu]\ar@{-}[ruu]\ar@{-}[rruu]\ar@{-}[rrr]\ar@{-}[rdd]\ar@{-}[dd]\ar@{-}[ldd]&&&*=0{\bullet}\ar@{-}[uu]\\ 
&&&&\\
*=0{\bullet}&*=0{\bullet}&*=0{\bullet}&&
}\quad 
\xymatrix@=4pt{  
*=0{\bullet}&&*=0{\bullet}&*=0{\bullet}\\          
&&&\\
&*=0{\bullet}\ar@{-}[luu]\ar@{-}[ruu]\ar@/^/@{-}[rr]\ar@/_/@{-}[rr]\ar@{-}[rdd]\ar@{-}[ldd]&&*=0{\bullet}\ar@{-}[uu]\ar@{-}[dd]\\
&&&\\ 
*=0{\bullet}&&*=0{\bullet}&*=0{\bullet}
}\quad 
\xymatrix@=4pt{  
*=0{\bullet}&*=0{\bullet}&*=0{\bullet}&*=0{\bullet}&*=0{\bullet}&*=0{\bullet}\\
&&&&&\\ 
&*=0{\bullet}\ar@{-}[luu]\ar@{-}[uu]\ar@{-}[ruu]\ar@{-}[rrr]&&&*=0{\bullet}\ar@{-}[luu]\ar@{-}[uu]\ar@{-}[ruu]&
}\quad 
\xymatrix@=4pt{
*=0{\bullet}&*=0{\bullet}&*=0{\bullet}&*=0{\bullet}\\
&&&\\
*=0{\bullet}\ar@{-}[uu]\ar@/^/@{-}[rr]\ar @/^0.2pc/ @{-} [rr] \ar @/_0.2pc/ @{-} [rr]\ar@{-}[rr]\ar@/_/@{-}[rr]&&*=0{\bullet}\ar@{-}[luu]\ar@{-}[uu]\ar@{-}[ruu]\ar@{-}[rdd]\ar@{-}[ldd]&\\
&&&\\
&*=0{\bullet}&&*=0{\bullet}\\
}\quad 
\xymatrix@=4pt{
*=0{\bullet}&*=0{\bullet}&*=0{\bullet}&*=0{\bullet}\\
&&&\\
*=0{\bullet}\ar@{-}[uu]\ar@/^/@{-}[rr]\ar@{-}[rr]\ar@/_/@{-}[rr]&&*=0{\bullet}\ar@{-}[luu]\ar@{-}[uu]\ar@{-}[ruu]\ar@{-}[dd]\ar@{-}[rdd]\ar@{-}[ldd]&\\
&&&\\
&*=0{\bullet}&&*=0{\bullet}\\
}      
\end{align*}   
Notice that other continued fraction expansions are the essentially the same as one of the above 6 graphs.  Namely,  
\begin{align*}
&17\sim\frac{17}{16},\quad \frac{17}{2}\sim\frac{17}{8}\sim\frac{17}{9}\sim\frac{17}{15},\quad 
\frac{17}{3}\sim\frac{17}{5}\sim\frac{17}{12}\sim\frac{17}{14},\quad\frac{17}{4}\sim\frac{17}{13},\\ 
&\frac{17}{6}\sim\frac{17}{11},\quad\frac{17}{7}\sim\frac{17}{10}
\end{align*}

\section{Examples in chemistry}  

Ethylene, acetone (or dimethyl sulfoxide) and diazene correspond to the continued fraction expansions 
$$
3+\frac{2}{3}=\frac{11}{3},\quad 3+\frac{2}{1}=\frac{5}{1}\quad\hbox{and}\quad 2+\frac{2}{2}=\frac{6}{2}\,, 
$$ 
respectively.  These topological indices are given by $11$, $5$ and $6$, respectively.  In fact, the structure of diazene can be explained by Pell-Lucas number $Q_3$.   

Acetylene can be written as $D_4(1,1,1,1;1,3,1)$, $D_3(2,1,1;3,1)$ (or $D_3(1,1,2;1,3)$) or $D_2(2,2;3)$. Then the corresponding continued fractions are 
$$
1+\cfrac{1}{1+\cfrac{3}{1+\cfrac{1}{1}}}=\frac{7}{5},\quad 2+\cfrac{3}{1+\cfrac{1}{1}}=\frac{7}{2}\quad\hbox{or}\quad 2+\frac{3}{2}=\frac{7}{2}\,. 
$$ 
In any case its topological index is given by $7$.  

For cyanogen, by the continued fraction expansion 
$$
1+\cfrac{3}{1+\cfrac{1}{1+\cfrac{3}{1}}}=\frac{17}{5}\,, 
$$ 
its topological index is given by $Z\bigl(D_4(1,1,1,1;3,1,3)\bigr)=17$.

\section{Awful graphs}   

Caterpillar-bond graphs associated with general continued fractions are not only mere extensions of caterpillar graphs with simple continued fractions, but also yield more availabilities. For example, by using several expressions of the same value by general continued fractions
$$
3+\cfrac{3}{1+\cfrac{4}{2+\cfrac{1}{4}}}=3+\cfrac{3}{1+\cfrac{16}{9}}=3+\frac{27}{25}=\frac{102}{25}\,, 
$$
the topological indices are given by 
$$
Z\bigl(D_4(3,1,2,4;3,4,1)\bigr)=Z\bigl(D_3(3,1,9;3,16)\bigr)=Z\bigl(D_2(3,25;27)\bigr)=102\,. 
$$ 
 
Although the appearance may be bad, the techniques used here are useful for calculating the topological index of more complex graphs. Details will be described in the following paper. 
 
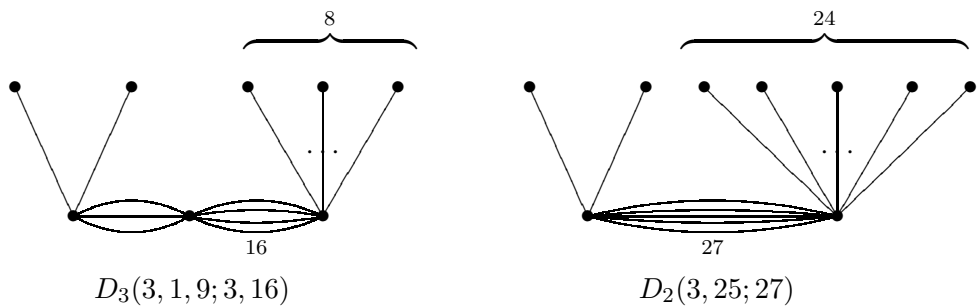
\begin{align*}  
&\qquad\qquad\qquad\qquad \overbrace{\phantom{\hspace{0.9in}}}^{8}\qquad\qquad\qquad\qquad\quad   \overbrace{\phantom{\hspace{1.5in}}}^{24}\\* 
&\xymatrix@=16pt{
*=0{\bullet}&&*=0{\bullet}&&*=0{\bullet}&*=0{\bullet}&*=0{\bullet}\\
&&&&&\cdots&\\ 
&*=0{\bullet}\ar@{-}[luu]\ar@{-}[ruu]\ar@/^/@{-}[rr]\ar@{-}[rr]\ar@/_/@{-} [rr]&&*=0{\bullet}\ar@/^/@{-}[rr]\ar @/^0.2pc/ @{-} [rr] \ar @/_0.2pc/ @{-} [rr]\ar@/_/@{-}[rr]_{16}&&*=0{\bullet}\ar@{-}[luu]\ar@{-}[uu]\ar@{-}[ruu]&  
}\qquad\qquad  
\xymatrix@=16pt{
*=0{\bullet}&&*=0{\bullet}&*=0{\bullet}&*=0{\bullet}&*=0{\bullet}&*=0{\bullet}&*=0{\bullet}\\
&&&&&\cdots&&\\ 
&*=0{\bullet}\ar@{-}[luu]\ar@{-}[ruu]\ar@/^/@{-}[rrrr]\ar@{-}[rrrr]\ar @/^0.2pc/ @{-} [rrrr] \ar @/_0.2pc/ @{-} [rrrr]\ar@/_/@{-}[rrrr]_{27}&&&&*=0{\bullet}\ar@{-}[lluu]\ar@{-}[luu]\ar@{-}[uu]\ar@{-}[ruu]\ar@{-}[rruu]&& 
}\\*   
&\qquad\quad D_3(3,1,9;3,16)\qquad\qquad\qquad\qquad\qquad\qquad D_2(3,25;27) 
\end{align*}

\section{Acknowledgments} 

The author thanks Haruo Hosoya for his directions and reprints. He is grateful for the discussion with Haruo Hosoya and Jin Akiyama for related problems.  


\begin{thebibliography}{99} 

\bibitem{DB} 
J. Devillers and A. T. Balaban, {\em  
Topological indices and related descriptors in QSAR and QSPR},  
Boca Raton, CRC Press, 2000. 
ISBN 90-5699-239-2.   

\bibitem{GP}  
I. Gutman and O. E. Polansky, {\em  
Mathematical Concepts in Organic Chemistry}, 
Springer, Berlin, 1986. 

\bibitem{Hosoya1971}   
H. Hosoya, {\em 
Topological index. A newly proposed quantity characterizing the topological nature of structural isomers of saturated hydrocarbons}, 
Bull. Chemical Soc. Japan, {\bf 44}.9 (1971), 2332--2339. 
DOI:10.1246/bcsj.44.2332

\bibitem{Hosoya1973}  
H. Hosoya, {\em  
Topological index and Fibonacci numbers with relation to Chemistry}, 
Fibonacci Quart. {\bf 11}.3 (1973), 255-266. 

\bibitem{Hosoya2007c}  
H. Hosoya, {\em 
Mathematical meaning and importance of the topological index $Z$},  
Croat. Chem. Acta (CCACAA) {\bf 80}.2 (2007), 239--249.    

\bibitem{Hosoya2007} 
H. Hosoya, {\em  
Continuant, caterpillar, and topological index $Z$. Fastest algorithm for degrading a continued fraction},  
Natural Sci. Rep. Ochanomizu Univ. {\bf 58}.1 (2007), 15--28.  
available at http://hdl.handle.net/10083/35234 

\bibitem{Hosoya2008}  
H. Hosoya, {\em 
Continuant, caterpillar, and topological index $Z$. II. Novel identities involving Fibonacci, Lucas, and generalized Fibonacci numbers},  
Natural Sci. Rep. Ochanomizu Univ. {\bf 58}.2 (2008), 11--20.  
available at http://hdl.handle.net/10083/35236 

\bibitem{Hosoya2010}  
H. Hosoya, {\em  
Continuant, caterpillar, and topological index $Z$. III. Graph-theoretical algorithm for and interpretation of solving linear Diophantine equations},  
Natural Sci. Rep. Ochanomizu Univ. {\bf 60}.2 (2010), 17--27.  
available at http://hdl.handle.net/10083/49184 

\bibitem{March}  
J. March, {\em  
Advanced Organic Chemistry: Reactions, Mechanisms, and Structure},  
(3rd ed.), New York, Wiley, 1985.  

\bibitem{oeis}  
N. J. A. Sloane,  {\em 
On-Line Encyclopedia of Integer Sequences}, 
available at oeis.org, (2019).  

\end{thebibliography}
\end{document}